\documentclass{amsart}

\usepackage{url}
\usepackage{amsmath,amsthm}

\theoremstyle{theorem}
\newtheorem{theorem}{Theorem}

\newtheorem{proposition}[theorem]{Proposition}
\newtheorem{lemma}[theorem]{Lemma}

\theoremstyle{definition}
\newtheorem{definition}[theorem]{Definition}

\begin{document}
\title{Combinatorics of Multicompositions}
\author{Brian Hopkins}
\address{Saint Peter's University, Jersey City NJ 07306, USA}
\email{bhopkins@saintpeters.edu}
\author{St\'ephane Ouvry}
\address{LPTMS, CNRS, Universit\'{e} Paris-Sud, 91405 Orsay Cedex, France}
\email{stephane.ouvry@u-psud.fr}

\maketitle             

\begin{abstract}
Integer compositions with certain colored parts were introduced by Andrews in 2007 to address a number-theoretic problem.  Integer compositions allowing zero as some parts were introduced by Ouvry and Polychronakos in 2019.  We give a bijection between these two varieties of compositions and determine various combinatorial properties of these multicompositions.  In particular, we determine the count of multicompositions by number of all parts, number of positive parts, and number of zeros.  Then, working from three types of compositions with restricted parts that are counted by the Fibonacci sequence, we find the sequences counting multicompositions with analogous restrictions.  With these tools, we give combinatorial proofs of summation formulas for generalizations of the Jacobsthal and Pell sequences.
\end{abstract}

\section{Introduction}
A composition of a positive integer $n$ is an ordered collection of positive integers whose sum is $n$.  For instance, there are four compositions of 3: $1+1+1$, $1+2$, $2+1$, and $3$.  We refer to the summands as parts.  

The next section begins with the formal definition of multicompositions that we will use.  Here are two closely related generalizations of compositions that inform the current work.

In 2007, George Andrews introduced $k$-compositions where $k$ copies of each positive integer are available as parts, denoted by subscripts, sometimes considered as colors \cite{a}.  There is a restriction that the last part must have subscript 1.  There are nine 2-compositions of 3:
\begin{gather*} 
1_1+1_1+1_1, 1_1+1_2+1_1, 1_2+1_1+1_1, 1_2+1_2+1_1,\\
 1_1+2_1, 1_2+2_1, 2_1+1_1, 2_2+1_1, 3_1.
 \end{gather*}
Andrews used these multicompositions to generalize and solve a problem of Emeric Deutsch about divisors of the number of compositions with relatively prime summands, a variety of compositions introduced by Henry Gould.

In 2019, the second named author and Alexios Polychronakos introduced $g$-compositions where up to $g-2$ zeros can occur between positive parts \cite[p. 11]{op}.  Thus standard compositions are the $g=2$ case.  For $g=3$, there are nine such compositions of 3:
\begin{gather*}
1+1+1, 1+1+0+1, 1+0+1+1, 1+0+1+0+1, \\
 1+2, 1+0+2, 2+1, 2+0+1, 3. 
 \end{gather*}
Here the motivation comes from quantum mechanics and the notion of exclusion statistics.  See the last section for a description along with an open problem.

In Section \ref{sec2}, we give our definition of multicompositions and determine their count and generating function, introducing a combinatorial interpretation that is the basis for subsequent results.  Section \ref{sec3} gives counts of multicompositions by various kinds of parts: number of all parts, number of positive parts, and number of zeros.  Section \ref{sec4} is informed by three types of restricted standard compositions counted by the Fibonacci sequence; the analogous multicomposition counts diverge into different families of sequences.  In Section \ref{sec5}, we connect the triangular arrays of numbers from Section \ref{sec3} with sequences from Section \ref{sec4} and a different kind of of composition introduced in 2020 \cite{bhr}---these results allow us to give combinatorial proofs of summation formulas for generalizations of Jacobsthal and Pell sequences.  Section \ref{sec6} gives more background for the physics motivation \cite{op,ow} and some open questions.

\section{Multicompositions and generating functions} \label{sec2}
Now we give our two-part definition of multicompositions, explain how these connect to the earlier definitions, and show that the two manifestations are equivalent.

\begin{definition} \label{kcompdef}
A $k$-composition of $n$ is a standard composition of $n$ with either of the following equivalent modifications.
\begin{itemize}
\item[(a)] Each part is assigned a color $1, \ldots, k$ (denoted by a subscript) except that the first part must have color 1.
\item[(b)] Each part except the first can be immediately preceded by up to $k-1$ zeros.
\end{itemize}
Write $C^k(n)$ for the set of $k$-compositions of $n$ and $c^k(n)$ for the number of them.  Collectively, or when $k$ is not specified, these compositions are called multicompositions.
\end{definition}

The connections to the earlier definitions are clear: The colored parts definition (a) just switches the restriction of Andrews's $k$-composition from requiring the last part to have color 1 to the first part.  The internal zeros definition (b) is precisely the Ouvry--Polychronakos definition with $k = g-1$.  It remains for us to connect our (a) and (b) definitions.

We show that the (a) and (b) definitions of multicompositions are equivalent by demonstrating a bijection between the compositions they describe.  The correspondence is determined part by part, explicitly
\[c_\ell \longleftrightarrow \overbrace{0+\cdots+0}^{\ell-1 \text{ zeros}}+c.\]
The correspondence for the two versions of 2-compositions of 3 is shown in Table \ref{3bij}
.  With the equivalence of the two definitions established, we will switch between them as needed.
\begin{table}[t]  
\centering
{\renewcommand{\arraystretch}{1.3}
\begin{tabular}{c|c} \hline
colored parts \, & \, internal zeros \\ \hline
$1_1+1_1+1_1$ & 1+1+1\\
$1_1+1_1+1_2$ & 1+1+0+1\\
$1_1+1_2+1_1$ & 1+0+1+1\\ 
$1_1+1_2+1_2$ & 1+0+1+0+1\\
$1_1+2_1$ & 1+2 \\ 
$1_1+2_2$ & 1+0+2\\
$2_1+1_1$ & 2+1 \\
$2_1+1_2$ & 2+0+1\\
$3_1$ & 3 \\ \hline
\end{tabular}}
\caption{Correspondence between the two versions of 2-compositions of 3.}
\label{3bij}
\end{table}

All subsequent references to $k$-compositions indicate the parametrization of Definition \ref{kcompdef}.  In Section \ref{sec6}, while discussing the motivation for $g$-compositions, we will return to the Ouvry--Polychronakos indexing, where the $g$ parameter is clearly indicated.

Next, we determine $c^k(n)$, the number of $k$-compositions of $n$, and its generating function.
\begin{proposition} \label{p2}
The generating function for $c^k(n)$ is
\begin{align}
\sum_{n \ge 1} c^k(n) x^n & = \frac{\sum_{i\ge1} x^i}{1-k\sum_{i\ge1} x^i} \label{e1} \\
				     & = \frac{x}{1-(k+1)x} \label{e2}
\end{align}
and $c^k(n) = (k+1)^{n-1}$.
\end{proposition}
\begin{proof}
The expression \eqref{e1} follows from the colored parts definition of multicompositions: The numerator accounts for the first part which must have color 1 while the denominator indicates a geometric series where each subsequent part can have color $1, \ldots$, or  $k$.  The simplification to \eqref{e2} can be seen with the intermediate steps
\[\frac{x\left(\frac{1}{1-x}\right)}{1-kx \left(\frac{1}{1-x}\right)} = \frac{x}{1-x-kx}.\]

The $c^k(n)$ formula follows directly from the generating function.  However, we give a combinatorial proof instead, in order to introduce ideas used in the sequel.

To construct a $k$-composition of length $n$, consider starting with a length $n$ board.  Working from left to right, at each of the $n-1$ internal integer positions, make a choice about the relation of the two adjacent squares: they can either be joined into a longer part or separated so that a new part starts to the right.  Further, a new part is assigned one of $k$ colors.  Thus there are $k+1$ choices at each internal position, denoted by markers J for join, ${\rm S}_1$ for separate and start a part with color 1, \dots, ${\rm S}_k$ for separate and start a part with color $k$.  Figure \ref{4comps} 
shows several examples of these tilings and the corresponding 2-compositions of 3.
Note that there is no separation choice for the first part; its color is always 1.  It is clear that this is a bijection between the choices and $k$-compositions of $n$.  The number of possibilities for $n-1$ choices each with $k+1$ options is $(k+1)^{n-1}$.
\end{proof}
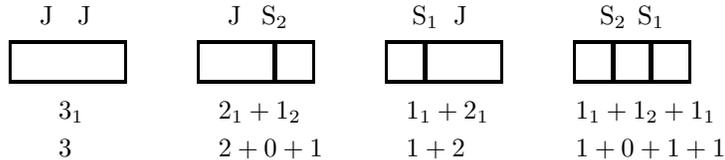
\begin{figure}[t] 
\centering
\setlength{\unitlength}{.5cm}
\begin{picture}(18,4)
\setlength{\fboxsep}{0pt}
\thicklines
\put(0.75,3.5){J}
\put(1.75,3.5){J}
\put(0,2){\framebox(3,1){}}
\put(1.25,1){$3_1$}
\put(1.25,0){$3$}
\put(5.75,3.5){J}
\put(6.65,3.5){${\rm S}_2$}
\put(5,2){\framebox(2,1){}}\put(7,2){\framebox(1,1){}}
\put(5.5,1){$2_1+1_2$}
\put(5.5,0){$2+0+1$}
\put(10.65,3.5){${\rm S}_1$}
\put(11.75,3.5){J}
\put(10,2){\framebox(1,1){}}\put(11,2){\framebox(2,1){}}
\put(10.5,1){$1_1+2_1$}
\put(10.5,0){$1+2$}
\put(15.65,3.5){${\rm S}_2$}
\put(16.65,3.5){${\rm S}_1$}
\put(15,2){\framebox(1,1){}}\put(16,2){\framebox(1,1){}}\put(17,2){\framebox(1,1){}}
\put(15,1){$1_1 + 1_2 + 1_1$}
\put(15,0){$1+0+1+1$}
\end{picture}
\caption{Four sequences of two ordered choices from the options J, ${\rm S}_1$, ${\rm S}_2$, and the corresponding 2-compositions of 3 (with both definitions).}
\label{4comps}
\end{figure}

The simplicity of $c^k(n) = (k+1)^{n-1}$ suggests that these multicompositions are a compelling generalization of standard compositions.

The $k=1$ case gives the $2^{n-1}$ standard compositions of $n$ and here the bijection reduces to a combinatorial argument given by MacMahon; see \cite{hm} for additional background.  Andrews establishes the $c^k(n)$ value also starting from MacMahon's bijection, but he proceeds in a different way \cite[Lemma 3]{a}.  

Note that our generating function starts from $n=1$, rather than $n=0$, as this leads to simpler expressions later and avoids requiring us to ponder the empty composition.

The bijection given in the proof is compatible with the internal zeros definition as well: ${\rm S}_m$ with $1 \le m \le k$ corresponds to inserting $m-1$ zeros in the sum before starting the next positive part.  Figure \ref{4comps} 
shows the multicompositions with both definitions.

\section{Counting multicompositions by various parts} \label{sec3}
Counting the number of multicompositions with a given number of parts is a more nuanced undertaking than for standard compositions.  Using the internal zeros definition of multicompositions, one could be interested in the number of all parts, the number of positive parts, or the number of zeros.  In this section, we address each of these variants in turn.  For each, we determine both a recurrence relation and a direct formula.  These results provide new combinatorial interpretations for several known integer sequences.  

Also, each subsection includes triangles of numbers for the $k=2,3$ cases, where row $n$ partitions the $k$-compositions of $n$ and the columns reference the number of various parts.  Therefore the row sums are $(k+1)^{n-1}$.

First, we show that these counts have an interpretation in the colored parts definition of multicompositions, although it is less a natural statistic in that setting.
\begin{lemma} \label{colorzero}
Comparing a multicomposition under the two definitions, the number of zeros equals the sum of $(\text{color } - 1)$ over the colored parts.
\end{lemma}
\begin{proof}
In the bijection of the proof of Proposition \ref{p2}, each part beyond the first starts with a marker ${\rm S}_m$.  Depending on the definition in use, this contributes $m-1$ zeros before the next part or colors the next part $m$.  The first part contributes no zeros and must have color 1.  
\end{proof}

\subsection{Number of all parts}
A standard composition of $n$ has up to $n$ parts.  Since a $k$-composition of $n$ can have up to $k-1$ zeros in up to $n-1$ positions between successive positive parts, it can have up to $nk - k + 1$ parts.  Write $C^k(n,\ell)$ for the set of $k$-compositions of $n$ with $\ell$ parts and $c^k(n,\ell)$ for the number of them.  Triangles of these values for $k = 2,3$ through $n=4$ are given in Table \ref{totalTnl}
; the leftmost column in each triangle corresponds to $\ell=1$ since every $k$-composition has at least one part.
\begingroup
\begin{table}[t] 
\renewcommand{\arraystretch}{1.3}
\setlength{\tabcolsep}{5pt}
\begin{center}
(a) \begin{tabular}{ccccccc}
1 \\
1 & 1 & 1  \\
1 & 2 & 3 & 2 & 1  \\
1 & 3 & 6 & 7 & 6 & 3 & 1 
\end{tabular} 
(b) \begin{tabular}{cccccccccc}
1 \\
1 & 1 & 1 & 1 \\
1 & 2 & 3 & 4 & 3 & 2 & 1 \\
1 & 3 & 6 & 10 & 12 & 12 & 10 & 6 & 3 & 1
\end{tabular}
\end{center}
\caption{Triangles of $c^k(n,\ell)$ values where the row matches the sum $n$ and the column indicates the number of parts $\ell$, for (a) $k=2$ and (b) $k=3$.  These are the trinomial \cite[A027907]{o} and quadrinomial \cite[A008287]{o} coefficients, respectively.}
\label{totalTnl}
\end{table}
\endgroup

Let $\binom{n}{\ell}_k = [x^\ell](1+x+\cdots+x^k)^n$ where $[x^\ell]$ denotes the coefficient of $x^\ell$ in the subsequent polynomial.  These are known as multinomial (or polynomial) coefficients and generalize binomial coefficients, the $k=1$ case.  
\begin{proposition} \label{allparts}
The number of $k$-compositions of $n$ with $\ell$ parts satisfies the recurrence
\[c^k(n,\ell)  = c^k(n-1,\ell-k) + \cdots + c^k(n-1,\ell-1)+ c^k(n-1,\ell)\]
and $c^k(n,\ell) = \binom{n-1}{\ell - 1}_k$.
\end{proposition}
\begin{proof}
For the recurrence, we establish a bijection between the multicompositions counted on each side of the equation.  Given a $k$-composition in the set $C^k(n-1,\ell-j)$ with $1 \le j \le k$, append $j$ terms, namely $j-1$ zeros and a 1.  Given a $k$-composition in $C^k(n-1,\ell)$, increase the last part by 1.  Both of these constructions give $k$-compositions with sum $n$ and $\ell$ parts.  For the inverse, if the last part of a $k$-composition in $C^k(n,\ell)$ is greater than 1, decrease it by 1; otherwise remove the final 1 and any zeros between it and the previous positive part.

For the direct formula, consider $(1+x+\cdots+x^k)^{n-1}$: The choice of summand for each factor corresponds to putting a J (from the choice 1) or ${\rm S}_m$ (from the choice $x^m$) in each of the $n-1$ internal positions of a length $n$ board.  Each $x^m$ then contributes $m$ parts to the resulting multicomposition: $m-1$ zeros and a new positive part.  Therefore, remembering the initial part that has no J or ${\rm S}_m$ marker, $\binom{n-1}{\ell-1}_k = [x^{\ell-1}](1+x+\cdots+x^k)^{n-1}$ gives the count for $k$-compositions of $n$ with $\ell$ parts.
\end{proof}

For a more combinatorial connection to multinomial coefficients, one can use Comtet's description that $\binom{n}{\ell}_k$ counts length $\ell$ multisets of $n$ letters each appearing at most $k$ times \cite[p. 77]{co}.  It is not difficult to describe a bijection between these multisets and the multicomposition tilings with J and ${\rm S}_m$ markers described in the proof of Proposition \ref{p2}.

\subsection{Number of positive parts}
Like standard compositions, a multicomposition of $n$ can have up to $n$ positive parts.  But with various arrangements of uncounted zeros, there are more possibilities.  Write $C_+^k(n,\ell)$ for the set of $k$-compositions of $n$ with $\ell$ positive parts, $c_+^k(n,\ell)$ for the count.  Triangles of these values for $k = 2,3$ through $n=5$ are given in Table \ref{posTnl}
; the leftmost column in each triangle corresponds to $\ell=1$ since every $k$-composition has at least one positive part.
\begingroup
\begin{table}[t] 
\renewcommand{\arraystretch}{1.3}
\setlength{\tabcolsep}{5pt}
\begin{center}
(a) \begin{tabular}{ccccc}
1 \\
1 & 2 \\
1 & 4 & 4 \\
1 & 6 & 12 & 8 \\
1 & 8 & 24 & 32 & 16
\end{tabular} 
(b) \begin{tabular}{ccccc}
1 \\
1 & 3 \\
1 & 6 & 9 \\
1 & 9 & 27 & 27 \\
1 & 12 & 54 & 108 & 81
\end{tabular} 
\end{center}
\caption{Triangles of $c_+^k(n,\ell)$ values where the row matches the sum $n$ and the column indicates the number of positive parts $\ell$, for (a) $k=2$ and (b) $k=3$.  These match \cite[A013609, A013610]{o}, respectively.}
\label{posTnl}
\end{table}
\endgroup
\begin{proposition} \label{posparts}
The number of $k$-compositions of $n$ with $\ell$ positive parts satisfies the recurrence
\[c_+^k(n,\ell)  = k c_+^k(n-1,\ell-1)+ c_+^k(n-1,\ell)\]
and $c_+^k(n,\ell) = k^{\ell-1} \binom{n-1}{\ell - 1}$.
\end{proposition}
\begin{proof}
Here we use the colored parts definition of multicompositions, where each part is positive.  

For the recurrence, given a $k$-composition in $C_+^k(n-1,\ell)$, increase the last part by 1.  Given a multicomposition in $C_+^k(n-1,\ell-1)$, there are $k$ possibilities for adding a 1, from $1_1$ to $1_k$.  These constructions give a $k$-composition with sum $n$ and $\ell$ (positive) parts.  The inverse is clear, conditioned on whether the last part is greater than 1.

For the direct formula, we count the ways to construct a $k$-composition of $n$ with $\ell$ parts from a length $n$ board.  To have $\ell$ parts, we need to place an ${\rm S}_m$ marker at $\ell-1$ of the $n-1$ internal positions; the positions can be chosen in $\binom{n-1}{\ell-1}$ ways.  There are $k$ choices of $m$ for each ${\rm S}_m$, giving the $k^{\ell-1}$ factor. 

We can also modify $c^k(n,\ell) = [x^{\ell-1}](1+x+\cdots+x^k)^{n-1}$ to determine the direct formula.  Since we are not counting zeros now, each $x^i$ contributes one positive part regardless of $i$. Therefore $c_+^k(n,\ell) = [x^{\ell-1}](1+kx)^{n-1}$ and the result follows from the binomial theorem.
 \end{proof}

\subsection{Number of zeros}
A multicomposition of $n$ can have up to $(n-1)(k-1)$ zeros.  Write $C_0^k(n,\ell)$ for the set of $k$-compositions of $n$ with $\ell$ zeros, $c_0^k(n,\ell)$ for the count.  Triangles of these values for $k = 2,3$ through $n=5$ are given in Table \ref{zeroTnl}
; the leftmost column in each triangle corresponds to $\ell=0$ since not all $k$-compositions includes zeros.  Since multicompositions with no zeros are the standard compositions, $c_0^k(n,0) = 2^{n-1}$.
\begingroup
\begin{table}[t] 
\renewcommand{\arraystretch}{1.3}
\setlength{\tabcolsep}{5pt}
\begin{center}
(a) \begin{tabular}{ccccc}
1 \\
2 & 1 \\
4 & 4 & 1 \\
8 & 12 & 6 & 1 \\
16 & 32 & 24 & 8 & 1
\end{tabular} 
(b) \begin{tabular}{ccccccccc}
1 \\
2 & 1 & 1 \\
4 & 4 & 5 & 2 & 1 \\
8 & 12 & 18 & 13 & 9 & 3 & 1 \\
16 & 32 & 56 & 56 & 49 & 28 & 14 & 4 & 1
\end{tabular} 
\end{center}
\caption{Triangles of $c_0^k(n,\ell)$ values where the row indicates the sum $n$ and the column indicates the number of zeros $\ell$, for (a) $k=2$ and (b) $k=3$.  These match \cite[A038207, A336996]{o}, respectively.}
\label{zeroTnl}
\end{table}
\endgroup
\begin{proposition} \label{zeroparts}
The number of $k$-compositions of $n$ with $\ell$ zeros satisfies the recurrence
\[c_0^k(n,\ell)  = c_0^k(n-1,\ell-k+1) + \cdots + c_0^k(n-1,\ell-1)+ 2c_0^k(n-1,\ell)\]
and 
\[c_0^k(n,\ell) = \sum_{m=0}^{n-1} \binom{n-1}{m} \binom{m}{\ell}_{k-1}.\]
\end{proposition}
\begin{proof}
For the recurrence, given a $k$-composition in the set $C_0^k(n-1,\ell-j)$ with $1 \le j \le k-1$, append $j$ zeros and then a 1.  Given a $k$-composition in $C_0^k(n-1,\ell)$, construct two multicompositions: increase the last part by 1 for one and append a 1 (with no zeros) for the other.  These constructions give $k$-compositions with sum $n$ and $\ell$ zeros.  For the inverse, if the last part of a $k$-composition in $C_0^k(n,\ell)$ is greater than 1, decrease it by 1; otherwise remove the final 1 and any zeros between it and the previous positive part.

For the direct formula, we modify $c^k(n,\ell) = [x^{\ell-1}](1+x+\cdots+x^k)^{n-1}$.  Recall that the marker ${\rm S}_m$ placed on a length $n$ board leads to $m-1$ zeros; neither J nor ${\rm S}_1$ contribute any zeros.  We recast the polynomial in $z$ to highlight counting zeros, where $x^m$ becomes $z^{m-1}$, and $1+x$ (for J and ${\rm S}_1$) becomes $2z^0$.  Thus $c_0^k(n,\ell) = [z^\ell](2+z+\cdots+z^{k-1})^{n-1}$.  To determine that coefficient, write $(2+z+\cdots+z^{k-1})^{n-1}$ as $(1+(1+z+\cdots+z^{k-1}))^{n-1}$, apply the binomial theorem, and use the definition of multinomial coefficients.  
\end{proof}

\section{Restricted part multicompositions} \label{sec4}
There are many interesting classes of standard compositions that arise by restricting which parts are allowed.  We review three such restrictions that lead to Fibonacci sequence counts and then explore their analogues for multicompositions.

Recall the Fibonacci sequence $F(0) = 0$, $F(1) = 1$, and $F(n) = F(n-1)+F(n-2)$ for $n \ge 2$.  We focus on three types of restricted compositions of $n$: those made with just parts 1 and 2, denoted $C_{12}(n)$; those with only odd parts, $C_{\text{odd}}(n)$; and those where 1 is not allowed as a part, $C_{\hat{1}}(n)$.  There will be a superscript $k$ added for $k$-compositions, where zeros satisfying the definition are also allowed.  Again, lower case letters refer to the sizes of the corresponding sets.
\begin{proposition} For $n \ge 1$, the numbers of restricted compositions of $n$ as defined above are $c_{12}(n) = F(n+1)$, $c_{\text{odd}}(n) = F(n)$, and $c_{\hat{1}}(n) = F(n-1)$.
\end{proposition}

Proofs of these counts will follow from the $k=1$ cases of the following results on restricted multicompositions.  The history here is interesting.  The connection between Fibonacci numbers and sums using just 1 and 2 dates from ancient India in the study of poetry meters in languages with long and short vowels, the long vowels twice the length of short vowels \cite{s}.  De Morgan discussed compositions with only odd parts in 1846 \cite[Appendix 10]{d} and Cayley documented compositions with 1 prohibited in 1876 \cite{ca}.
\begin{proposition} \label{seqdiv}
Let a positive integer $k$ be given.
\begin{enumerate}
\item[(a)] The number of $k$-compositions of $n$ with only positive parts 1 and 2 satisfies 
\[c^k_{12}(1) = 1, c^k_{12}(2) = k+1, \text{and } c^k_{12}(n) =  k c^k_{12}(n-1) + k c^k_{12}(n-2) \text{ for } n \ge 3.\]
\item[(b)] The number of $k$-compositions of $n$ with only positive parts odd satisfies 
\[c^k_{\text{odd}}(1) = 1, c^k_{\text{odd}}(2) = k, \text{and } c^k_{\text{odd}}(n) =  k c^k_{\text{odd}}(n-1) + c^k_{\text{odd}}(n-2) \text{ for } n \ge 3.\]
\item[(c)] The number of $k$-compositions of $n$ with no parts 1 satisfies 
\[c^k_{\hat{1}}(2) = 1, c^k_{\hat{1}}(3) = 1, \text{and } c^k_{\hat{1}}(n) =  c^k_{\hat{1}}(n-1) + k c^k_{\hat{1}}(n-2) \text{ for } n \ge 4.\]
\end{enumerate}
\end{proposition}
\begin{proof}
We use here the colored parts definition of multicompositions.  Recall that the first part must have color 1.

(a) For multicompositions with only parts 1 and 2, the only possibility for sum 1 is $1_1$.  For sum 2, there are the $k+1$ possibilities, $2_1$, $1_1 + 1_1$, \dots, $1_1 + 1_k$.  The recurrence follows from the generating function, a modification of \eqref{e1}:
\[ \sum_{n \ge 1} c^k_{12}(n) x^n = \frac{x+x^2}{1-k(x+x^2)}.\]
Or one can give a combinatorial argument: $C_{12}^k(n)$ consists of each element of $C^k_{12}(n-2)$ with $2_1$, \dots, or $2_k$ appended and each element of $C^k_{12}(n-1)$ with $1_1$, \dots, or $1_k$ appended.  Details of verifying the bijection are left to the reader.

(b) For multicompositions with only odd parts, the only possibility for sum 1 is $1_1$.  For sum 2, there are the $k$ possibilities $1_1 + 1_1$, \dots, $1_1 + 1_k$.  The recurrence follows from the generating function
\[\sum_{n \ge 1} c^k_{\text{odd}}(n) x^n = \frac{x + x^3 + x^5 + \cdots}{1-k(x - x^3 - x^5 - \cdots)}
 = \frac{x\left(\frac{1}{1-x^2}\right)}{1-kx\left(\frac{1}{1-x^2}\right)}  = \frac{x}{1-kx-x^2} \]
or a combinatorial argument: $C_{\text{odd}}^k(n)$ consists of each element of $C^k_{\text{odd}}(n-2)$ with the last part extended by 2 (i.e., converted to the next larger odd number while maintaining the same color) and each element of $C^k_{\text{odd}}(n-1)$ with $1_1$, \dots, or $1_k$ appended.

(c) For multicompositions with 1 prohibited, the only possibility for sum 2 is $2_1$ and the only possibility for sum 3 is $3_1$.  The recurrence follows from the generating function
\[\sum_{n \ge 1} c^k_{\hat{1}}(n) x^n = \frac{x^2 + x^3 + x^4 + \cdots}{1-k(x^2 - x^3 - x^4 - \cdots)} 
 = \frac{x^2\left(\frac{1}{1-x}\right)}{1-kx^2\left(\frac{1}{1-x}\right)}  = \frac{x^2}{1-x-kx^2}\] 
or a combinatorial argument: $C_{\hat{1}}^k(n)$ consists of each element of $C^k_{\hat{1}}(n-2)$ with $2_1$, \dots, or $2_k$ appended and each element of $C^k_{\hat{1}}(n-1)$ with the last part extended by 1 (i.e., converted to the next larger integer, at least 3, while maintaining the same color).
\end{proof}

Thus these three restrictions, which are all counted by the Fibonacci numbers (with different indices) for standard compositions, diverge into three families of integer sequences when applied to multicompositions.  Initial terms of each sequence for $k=2,3,4$ are given in Table \ref{resseq}
.
\begingroup
\begin{table}
\renewcommand{\arraystretch}{1.3}
\setlength{\tabcolsep}{5pt}
\begin{center}
\begin{tabular}{c|cccccccc|c} \hline
sequence\textbackslash$n$ & 1 & 2 & 3 & 4 & 5 & 6 & 7 & 8 & name \\ \hline
$c^2_{12}(n)$ & 1 & 3 & 8 & 22 & 60 & 164 & 448 & 1224 & \cite[A028859]{o} \\
$c^3_{12}(n)$ & 1 & 4 & 15 & 57 & 216 & 819 & 3105 & 11772  & \cite[A125145]{o} \\ 
$c^4_{12}(n)$ &1 & 5 & 24 & 116 & 560 & 2704 & 13056 & 63040 & \cite[A086347]{o} \\ \hline
$c^2_{\text{odd}}(n)$ & 1 & 2 & 5 & 12 & 29 & 70 & 169 & 408 & Pell \\
$c^3_{\text{odd}}(n)$ & 1 & 3 & 10 & 33 & 109 & 360 & 1189 & 3927 & \cite[A006190]{o} \\ 
$c^4_{\text{odd}}(n)$ &1 & 4 & 17 & 72 & 305 & 1292 & 5473 & 23184 & \cite[A001076]{o} \\ \hline
$c^2_{\hat{1}}(n)$ & 0 & 1 & 1 & 3 & 5 & 11 & 21 & 43 & Jacobsthal \\
$c^3_{\hat{1}}(n)$ & 0 & 1 & 1 & 4 & 7 & 19 & 40 & 97 & \cite[A006130]{o} \\ 
$c^4_{\hat{1}}(n)$ & 0 & 1 & 1 & 5 & 9 & 29 & 65 & 181 & \cite[A006131]{o}\\ \hline
\end{tabular} 
\end{center}
\caption{Restricted multicomposition counts $c^k_{12}(n)$, $c^k_{\text{odd}}(n)$, and  $c^k_{\hat{1}}(n)$ for $2 \le k \le 4$ and $1 \le n \le 8$  with sequence identifiers.}
 \label{resseq}
\end{table}
\endgroup

\section{Connections to diagonal sums of triangles} \label{sec5}
A popular identity demonstrates how the Fibonacci numbers are ``inside'' Pascal's triangle: the finite sum $\sum_{i\ge0} \binom{n-i}{i} = F(n-1)$; see \cite[Identity 4]{bq} for a combinatorial proof that makes use of the restricted compositions $C_{12}(n)$.  A direct generalization shows that the diagonal sums of the $\binom{n}{\ell}_2$ triangle shown Table \ref{totalTnl}(a) 
are the tribonacci numbers \cite{hb}, etc.  It is not hard to give a bijection between the 2-compositions counted by $\sum_{i\ge0} \binom{n-i}{i}_2$ and the standard compositions with only parts 1, 2, and 3, etc.

In this section we give two results on the sum of diagonal entries of the other triangles in Section \ref{sec3}.

One can check that the diagonal sums of the $c_+^2(n,\ell)$ and $c_+^3(n,\ell)$ triangles in Table \ref{posTnl} 
match the first few terms of the sequences $c^2_{\hat{1}}(n)$ and $c^3_{\hat{1}}(n)$ of Table \ref{resseq}
, respectively.  We show that this holds in general.  But first we give a name to that family of sequences.
\begin{definition}
For a given positive integer $k$, the $k$-Jacobsthal sequence has initial conditions $J^k(0) = 0$, $J^k(1) = 1$ and, for $n \ge 2$, the recurrence 
\[J^k(n) = J^k(n-1) + k J^k(n-2).\]
\end{definition}

Note that $J^1(n) = F(n)$ and $J^2(n)$ is the standard Jacobsthal sequence \cite[A001045]{o}.  Proposition \ref{seqdiv}(c) could now be written $c^k_{\hat{1}}(n) = J^k(n-1)$; Table \ref{resseq} 
shows initial terms of $J^3(n-1)$ and $J^4(n-1)$.  Realize that many different sequences in the literature are called generalized or $k$-Jacobsthal numbers.
\begin{theorem}
Given a positive integer $k$, there is a bijection between the multicompositions $C_+^k(n,1) \cup C_+^k(n-1,2) \cup \cdots$ and $C_{\hat{1}}^k(n+1)$.  Therefore, 
\[J^k(n) = \sum_{i\ge1} k^{i-1} \binom{n-1-i}{i-1}.\]
\end{theorem}
\begin{proof}
The bijection simply consists of increasing by 1 each positive part in all the multicompositions of $\bigcup_i C_+^k(n+1-i,i)$.  This takes a multicomposition with sum $n+1-i$ having $i$ positive parts to a multicomposition of $n+1$ with no parts 1 (as every part arose from adding 1 to a positive part).  For the inverse map, given a multicomposition in $C_{\hat{1}}^k(n+1)$ with $i$ positive parts, decrease each positive part by 1 to produce a multicomposition in $C_+^k(n+1-i,i)$.

The formula for $J^k(n)$ then follows from its association with $C^k_{\hat{1}}(n)$ in Proposition \ref{seqdiv}(c) and the formula for $c_+^k(n,\ell)$ in Proposition \ref{posparts}.  
\end{proof}

For our last result, one can check that the diagonal sums of $c_0^2(n,\ell)$ match the first few terms of the sequence $c^2_{\text{odd}}(n)$ of Table \ref{resseq}
, the Pell numbers.  However, the diagonal sums of $c_0^3(n,\ell)$ do not match $c^3_{\text{odd}}(n)$.  Instead, we show that the diagonal sums of $c_0^k(n,\ell)$ count a different variety of colored compositions associated with the Pell numbers that were introduced recently by Bravo, Herrera, and Ram\'{i}rez \cite{bhr}.
\begin{definition}
For a given positive integer $k$, the $k$-Pell sequence has initial conditions $P^k(-k+2) = \cdots = P^k(0) = 0$, $P^k(1)=1$ and, for $n \ge 2$, the recurrence 
\[P^k(n) = 2 P^k(n-1) + P^k(n-2) + \cdots + P^k(n-k).\]
\end{definition}

Note that $P^2(n)$ is the standard Pell sequence \cite[A000129]{o}.  The sequences for $k=3,4,5$ are \cite[A077939, A103142, A141448]{o}, respectively.  Bravo, Herrera, and Ram\'{i}rez establish a natural combinatorial interpretation: $P^k(n+1)$ counts compositions of $n$ with parts $1, 1', 2, 3, \ldots, k$, i.e., standard compositions except that there is a second type of 1 \cite[Theorem 11]{bhr}.  Let $B^k(n)$ be the set of these compositions of $n$.
\begin{theorem}
Given an integer $k \ge 2$, there is a bijection between the multicompositions $C_0^k(n,0) \cup C_0^k(n-1,1) \cup \cdots$ and $B^k(n-1)$.  Therefore
\[ P^k(n) = \sum_{i\ge0} \sum_{m=0}^{n-i} \binom{n-i}{m} \binom{m}{i}_{k-1}.\]
\end{theorem}
\begin{proof}
Write each multicompositions of $\bigcup_i C_0^k(n-i,i)$ in terms of the J and ${\rm S}_m$ markers introduced in the proof of Proposition \ref{p2} 
and then convert each J to $1'$ and each ${\rm S}_m$ to $m$.  This clearly gives a composition with parts from $1, 1', 2, 3, \ldots, k$; we need to show that its sum is $n-1$.

A multicomposition with sum $n-i$ and $i$ zeros has $n-i-1$ markers.  Suppose $j$ of those markers are J.  The remaining $n-i-j-1$ markers have the form ${\rm S}_m$ for $1 \le m \le k$.  By Lemma \ref{colorzero}, the sum of the $(m-1)$ from the various ${\rm S}_m$ must be $i$.  Therefore, the sum of the $m$ from the various ${\rm S}_m$ is $i+(n-i-j-1)=n-j-1$.  In the image composition, there are $j$ parts $1'$ which, with the total $n-j-1$ from the $m$ terms, give the sum $n-1$.

For the reverse map, suppose the composition of $n-1$ with parts from 1, $1'$, 2, 3, \ldots, $k$ has $h$ parts including $j$ that are $1'$.  Replace each $1'$ with J and each of the $h-j$ parts $m$ with a marker ${\rm S}_m$.  Converting this into a multicomposition with internal zeros, the sum of the resulting multicomposition is $h+1$.  Now the sum of the $m$ from the various ${\rm S}_m$ is $n-1-j$, so by Lemma \ref{colorzero} the number of zeros in the resulting multicomposition is the sum of the $(m-1)$ which is $n-1-j-(h-j)=n-1-h$.  That is, the multicomposition is in the set $C_0^k(h+1,n-1-h) = C_0^k(n-i,i)$ with the substitution $i=n-1-h$.

The formula for $P^k(n)$ comes from their association with the $B^k(n-1)$ compositions and the formula for $c_0^k(n,\ell)$ in Proposition \ref{zeroparts}.  
\end{proof}

\section{Exclusion statistics and further work} \label{sec6}
As a further combinatorial investigation, one could combine the ideas of Sections \ref{sec3} and \ref{sec4} by counting the restricted multicompositions $C^k_{12}(n)$, $C^k_{\text{odd}}(n)$, and $C^k_{\hat{1}}(n)$ by number of all parts, number of positive parts, and number of zeros, then look for patterns in their diagonal sums as in Section \ref{sec5} to find combinatorial derivations for additional formulas.

In the remainder of this final section, we explain how $g$-compositions arose and present some questions we still have about them.  From here until the end, we use the $g$-composition definition of \cite{op}; recall the connection to Definition \ref{kcompdef} that $g=k+1$.

In statistical mechanics (see, for example, \cite{m}), $n$-body partition functions encode the statistical equilibrium of a  system of $n$ particles at temperature $T$.
For $q$ quantum states, the Boltzmann factors $s(k) = \exp(-\beta\epsilon_k)$ for $1 \le k \le q$ are building blocks of the $n$-body partition functions where $\epsilon_k$ is a $1$-body energy state and $\beta=1/(k_\text{B}T)$ incorporates the Boltzmann constant $k_\text{B}$.  For our purposes, though, $\beta$ is irrelevant and can be set to 1.  Define a $n$-body partition function $Z(n)$ as the nested multiple sum
\[ Z(n)  = \sum_{k_1=1}^{q-2n+2} \sum_{k_2=1}^{k_1} \cdots \sum_{k_{n}=1}^{k_{n-1}}
s(k_1+2n-2) s(k_2+2n-4) \cdots s(k_{n-1}+2) s(k_{n}).\]
Due to the $+2$ shifts, the arguments of the Boltzman factors 
$s(k)$ in $Z(n)$  differ by at least 2. Terms with particles in adjacent energy states $\epsilon_k$ and $\epsilon_{k+1}$ are excluded: $Z(n)$ is
 the partition function for $n$ particles on the line  obeying exclusion
statistics of order $g=2$. For $g$-exclusion, the shifts are $+g$. 

Cluster coefficients $b(n)$ are determined by
\[ \log\left(\sum_{n=0}^{\infty}Z(n) z^n \right)=\sum_{n=1}^{\infty} b(n) z^n. \]
These were introduced to give a combinatorial enumeration of closed lattice walks with a given algebraic area \cite{op,ow}.
The $g$-compositions naturally arise when we solve for the cluster coefficients, e.g., 
\begin{gather*}
b(1)=\sum_{k=1}^q s^1(k), \\
 -b(2)={1\over 2}\sum_{k=1}^q s^2(k)+\sum_{k=1}^{q} s^1(k+1)s^1(k),  \\
b(3)={1\over 3}\sum_{k}^q s^3(k)+\sum_{k=1}^{q} s^2(k+1)s^1(k)+\sum_{k=1}^{q} s^1(k+1)s^2(k)\\+\sum_{k=1}^{q} s^1(k+2)s^1(k+1)s^1(k), 
\end{gather*}
with the understanding  $s(k)=0$ when $k>q$.  We have written the usually-implicit exponents 1 to highlight the connection to compositions.  Looking for example at the four sums which contribute to $b(3)$, notice that the exponents are precisely the four standard compositions of 3, which correspond to $g=2$. In the $g$-exclusion case, the exponents in $b(n)$ correspond to $g$-compositions of $n$, with $g$-exclusion manifesting in the insertion of up to $g-2$ zeros between positive parts, showing an even stronger exclusion between energy states. 

What about the coefficients that appear in the $b(n)$ expressions?  For instance, the 1/3, 1, 1, 1 in $b(3)$ for 2-exclusion above.  Suppose a $g$-composition $\ell$ of $n$ has parts $\ell_1, \ldots, \ell_j$ (subscripts here are indices, not colors) which may be positive integers or zeros as allowed by the definition.  The coefficient of the sum corresponding to $\ell$ is
\begin{align*}
c_g (\ell) & = {(\ell_1+\dots +\ell_{g-1}-1)!\over \ell_1! \cdots \ell_{g-1}!} \prod_{i=1}^{j-g+1} {\ell_i+\dots +\ell_{i+g-1}-1 \choose \ell_{i+g-1}} \\
& = {\prod_{i=1}^{j-g+1}(\ell_i + \dots + \ell_{i+g-1} -1)! \over \prod_{i=1}^{j-g} (\ell_{i+1} + \dots +\ell_{i+g-1} -1 )! }\prod_{i=1}^j {1\over \ell_i!},
\end{align*}
see \cite[p. 11]{op}.  Further, summing the $c_g (\ell)$ over all $g$-compositions of $n$ gives $\binom{gn}{n}/(gn)$.

What combinatorics, if any, is at work here?  For $g=2$, the binomial $\binom{2n}{n}$ counts the number of lattice  walks on a line starting from a given site and making $2n$ steps, $n$ to the right and $n$ to the left.  As in \cite{ow}, consider the half with the first step to the right.  For $n=3$, writing R for right and L for left, the ten walks are
\begin{itemize}
\item 1 with three right-left steps on top of each other (RLRLRL),
\item 3 with two right-left steps on top of each other followed by one right-left step (RRLLRL, RLRRLL, RLLRLR),
\item 3 with one right-left step followed by two right-left steps on top of each other (RLLRRL, RLRLLR, RRLRLL), and
\item 3 with three right-left steps following each other (RRRLLL, RRLLLR, RLLLRR);
\end{itemize}
see \cite[Figure 1]{ow}.  In general, the $2nc_2 (\ell)$ over all $g=2$ compositions $\ell$ count  the number of lattice walks on a line making $2n$ steps and having a given number of right-left steps in succession.

A combinatorial interpretation of $c_g (\ell)$ for $g$-compositions with $g\ge 3$ remains an open question.

\section*{Acknowledgements}
We would like to thank Olivier Giraud for assistance with the generating function arguments in Section 4.  Also, we appreciate a tip from Abdelmalek Abdesselam given through the site MathOverflow.  Finally, we commend the work of two conference organizers: Oleg Evnin of the Fourth Bangkok Workshop on Discrete Geometry, Dynamics, and Statistics, where this collaboration began, and Mel Nathanson of the Eighteenth Annual Workshop on Combinatorial and Additive Number Theory, where these results were (virtually) presented.

\end{document}